\documentclass{article}
\usepackage{amsmath}
\usepackage{amssymb}
\usepackage{amstext}
\usepackage{amsfonts}

\newtheorem{theorem}{Theorem}[section]

\newtheorem{lemma}[theorem]{Lemma}

\newtheorem{remark}[theorem]{Remark}

\newenvironment{proof}[1][Proof]{\noindent\textbf{#1.} }{\ \rule{0.5em}{0.5em}}
\numberwithin{equation}{section}

\begin{document}

\title{The Leray-Lions existence theorem\\
under general growth conditions}
\author{Giovanni Cupini$^{1}$, Paolo Marcellini$^{2}$\thanks{%
Corresponding author.}, Elvira Mascolo$^{2}$ \\
%EndAName
$\quad $ \\
{\normalsize $^{1}$Dipartimento di Matematica, Universit\`a di Bologna }\\
{\normalsize Piazza di Porta S. Donato 5, 40126 - Bologna, Italy }\\
{\normalsize giovanni.cupini@unibo.it }\\
{\normalsize $^{2}$Dipartimento di Matematica e Informatica ``U. Dini'',
Universit\`a di Firenze}\\
{\normalsize Viale Morgagni 67/A, 50134 - Firenze, Italy}\\
{\normalsize paolo.marcellini@unifi.it, elvira.mascolo@unifi.it}}
\date{}
\maketitle

\begin{abstract}
We prove an existence result of weak solutions $u\in W_{0}^{1,p}\left(
\Omega \right) \cap W_{\mathrm{loc}}^{1,q}\left( \Omega \right) $, to a
Dirichlet problem for a second order elliptic equation in divergence form,
under general and $p,q-$\textit{growth conditions} of the differential
operator. This is a first attempt to extend to general growth the well known
Leray-Lions existence theorem, which holds under the so-called natural
growth conditions with $q=p$. We found a way to treat the general context
with explicit dependence on $\left( x,u\right) $, other than on the gradient
variable $\xi =Du$; these aspects require particular attention due to the $%
p,q-$context, with some differences and new difficulties compared to the
standard case $p=q$.
\end{abstract}

%\QTP{Body Math}

%\QTP{Body Math}

\section{Introduction}

The celebrated existence theorem of weak solutions to a class of nonlinear
elliptic equations, published in 1965 by Jean Leray and Jacques-Louis Lions,
in the case of \textit{second order equations} is related to the Dirichlet
problem 
\begin{equation}
\left\{ 
\begin{array}{l}
\sum_{i=1}^{n}\frac{\partial }{\partial x_{i}}a^{i}\left( x,u\left( x\right)
,Du\left( x\right) \right) =b\left( x\right) ,\;\;\;\;\;x\in \Omega \,, \\ 
u=0\;\;\;\text{on }\partial \Omega \,,%
\end{array}%
\right.  \label{Dirichlet problem}
\end{equation}%
where $\Omega $ is an open and bounded set in $\mathbb{R}^{n}$ and $a\left(
x,u,\xi \right) =\left( a^{i}\left( x,u,\xi \right) \right) _{i=1,\ldots ,n}$
is a \textit{Carath\'{e}odory} vector field in $\Omega \times \mathbb{R}%
\times \mathbb{R}^{n}$; i.e. $a:\Omega \times \mathbb{R}\times \mathbb{R}%
^{n}\rightarrow \mathbb{R}^{n}$ is continuous with respect to $\left( u,\xi
\right) \in \mathbb{R}\times \mathbb{R}^{n}$ and measurable when $x$ varies
in $\Omega \subset \mathbb{R}^{n}$. The vector field $a$ satisfies the
growth condition 
\begin{equation}
\left\vert a\left( x,u,\xi \right) \right\vert \leq M\left( \left\vert \xi
\right\vert ^{p-1}+\left\vert u\right\vert ^{p-1}+b_{1}\left( x\right)
\right) ,  \label{Leray-Lions 1}
\end{equation}%
for a constant $M$, an exponent $p>1$, every $\left( u,\xi \right) \in 
\mathbb{R}\times \mathbb{R}^{n}$ and a function $b_{1}\in L^{p^{\prime
}}\left( \Omega \right) $, with $\frac{1}{p}+\frac{1}{p^{\prime }}=1$, as
well as the right hand side $b$ in (\ref{Dirichlet problem})$_{1}$: $b\in
L^{p^{\prime }}\left( \Omega \right) $. Moreover, for a.e. $x\in \Omega $
and $\left\vert u\right\vert $ bounded, it satisfies the coercivity 
\begin{equation}
\lim_{\left\vert \xi \right\vert \rightarrow +\infty }\frac{\left( a\left(
x,u,\xi \right) ,\xi \right) }{\left\vert \xi \right\vert +\left\vert \xi
\right\vert ^{p-1}}=+\infty \,  \label{Leray-Lions 2}
\end{equation}%
and the monotonicity condition 
\begin{equation}
\left( a\left( x,u,\xi \right) -a\left( x,u,\eta \right) ,\xi -\eta \right)
>0,\;\;\;\;\;\forall \;\xi ,\eta \in \mathbb{R}^{n}:\;\xi \neq \eta \,.
\label{Leray-Lions 3}
\end{equation}%
As usual $\left( a\left( x,u,\xi \right) ,\xi \right) $ in (\ref{Leray-Lions
2}) denotes the \textit{scalar product} in $\mathbb{R}^{n}$ of $a\left(
x,u,\xi \right) $ and $\xi $; and similarly in (\ref{Leray-Lions 3}). A 
\textit{weak solution} to the Dirichlet problem (\ref{Dirichlet problem}) is
a function $u:\Omega \rightarrow \mathbb{R}$ in the Sobolev space $%
W_{0}^{1,p}\left( \Omega \right) $ such that 
\begin{equation}
A\left( u,v\right) +\int_{\Omega }b\left( x\right) v\,dx=0\,,
\label{Leray-Lions 4}
\end{equation}%
where $A\left( u,v\right) $ is the form, linear in $v\in W_{0}^{1,p}\left(
\Omega \right) $, 
\begin{equation}
A\left( u,v\right) :=\int_{\Omega }\sum_{i=1}^{n}a^{i}\left( x,u\left(
x\right) ,Du\left( x\right) \right) v_{x_{i}}\,dx\;.  \label{Leray-Lions 5}
\end{equation}%
The existence result can be stated in the following way.

\begin{theorem}[{Leray-Lions \protect\cite[Th\'{e}or\`{e}me 2]{Leray-Lions
1965}}]
\label{Leray-Lions Theorem}Consider the Carath\'{e}odory vector field $%
a\left( x,u,\xi \right) =\left( a^{i}\left( x,u,\xi \right) \right)
_{i=1,\ldots ,n}$, $a:\Omega \times \mathbb{R}\times \mathbb{R}%
^{n}\rightarrow \mathbb{R}^{n}$ satisfying the growth conditions from above (%
\ref{Leray-Lions 1}) and from below (\ref{Leray-Lions 2}), and the
monotonicity condition (\ref{Leray-Lions 3}). If 
\begin{equation}
\lim_{\left\Vert u\right\Vert _{W_{0}^{1,p}\left( \Omega \right)
}\rightarrow +\infty }\frac{\left\vert A\left( u,u\right) \right\vert }{%
\left( \int_{\Omega }\left\vert Du\right\vert ^{p}\,dx\right) ^{1/p}}%
=+\infty \,  \label{Leray-Lions 6}
\end{equation}%
then the Dirichlet problem (\ref{Dirichlet problem}) has (at least) a weak
solution in $W_{0}^{1,p}\left( \Omega \right) $.
\end{theorem}

Note that assumption (\ref{Leray-Lions 6}) has the specific role to exclude,
for instance, eigenvalues problems, whose Dirichlet problem (\ref{Dirichlet
problem}) may lack the solution, unless the right hand side $b$ satisfies
some specific compatibility conditions. The same Leray-Lions \cite[Remarques
1$^{0}$]{Leray-Lions 1965} observed that a sufficient condition for the
validity of (\ref{Leray-Lions 6}) is 
\begin{equation}
\left( a\left( x,u,\xi \right) ,\xi \right) =\sum_{i=1}^{n}a^{i}\left(
x,u,\xi \right) \xi _{i}\,\geq m\left\vert \xi \right\vert ^{p}\,,
\label{Leray-Lions 7}
\end{equation}%
for constants $m>0$, $M\geq 0$ and for every $\xi =\left( \xi _{i}\right)
_{i=1,\ldots ,n}$ such that $\left\vert \xi \right\vert \geq M$.

We emphasize the relevance of Theorem \ref{Leray-Lions Theorem} when the
vector field $a\left( x,u,\xi \right) =\left( a^{i}\left( x,u,\xi \right)
\right) _{i=1,\ldots ,n}$ explicitly depend on $u$; otherwise the existence
result for the Dirichlet problem (\ref{Dirichlet problem}) simplifies,
because the theory of monotone operators can be more directly applied.

\bigskip

After the publication in 1965, several authors considered and extended in
some directions this existence result by Leray-Lions. First in 1966 the well
known article related to the equation $\sum_{i=1}^{n}\frac{\partial }{%
\partial x_{i}}a^{i}(Du)=b(u)$ by Hartman-Stampacchia \cite%
{Hartman-Stampacchia 1966}. %
%\begin{equation*}
%		\sum_{i=1}^{n}\frac{\partial }{\partial x_{i}}a^{i}\left(
%Du\left( x\right) \right) =b\left(u(x)\right) ,\;\;\;\;\;x\in \Omega \,. 
%\end{equation*}%
The same J.L.Lions, in his book \textit{Quelques m\'{e}thodes de r\'{e}%
solution des probl\`{e}mes aux limites non lin\'{e}aires} \cite[Th\'{e}or%
\`{e}me 2.8]{Lions 1969}, revisited the Leray-Lions theorem, essentially
with similar assumptions. Relevant are the generalizations due to
Boccardo-Murat-Puel \cite[Theorem 1]{Boccardo-Murat-Puel 1988},\cite[Theorem
2]{Boccardo-Murat-Puel 1992} under the so-called \textit{\textquotedblleft
natural growth conditions"}. In particular Boccardo-Murat-Puel considered in 
\cite{Boccardo-Murat-Puel 1988} obstacle unilateral variational solutions,
with the possibility for the obstacle to be identically $-\infty $, i.e. the
case corresponding to the variational equation. While in \cite%
{Boccardo-Murat-Puel 1992} Boccardo-Murat-Puel obtained an $L^{\infty }-$%
regularity theorem with $b\in L^{s_{0}}\left( \Omega \right) $ and $s_{0}>%
\frac{n}{p}$, which allowed them to study a more general right hand side $%
b=b(x,u,Du)$; in fact they considered what nowadays is named a \textit{%
Leray-Lions type} differential operator $A\left( u\right) $, of the form 
\begin{equation}
A\left( u\right) =\sum_{i=1}^{n}\frac{\partial }{\partial x_{i}}a^{i}\left(
x,u\left( x\right) ,Du\left( x\right) \right) +a_{0}\left( x,u\left(
x\right) ,Du\left( x\right) \right) \,,  \label{Leray-Lions type 1}
\end{equation}%
where $a_{0}\left( x,u,\xi \right) $ is \textit{regularizing nonlinear
first-order term}; i.e., a Carath\'{e}odory function defined in $\Omega
\times \mathbb{R}\times \mathbb{R}^{n}$ which satisfies \textit{"natural
growth conditions"} similar to $a\left( x,u,\xi \right) $ in (\ref%
{Leray-Lions 1}) and the \textit{coercivity condition} with respect to the $%
u-$variable 
\begin{equation}
a_{0}\left( x,u,\xi \right) u\geq m_{0}\left\vert u\right\vert
^{p}\,,\;\;\;\;\forall \;\left( x,u,\xi \right) \in \Omega \times \mathbb{R}%
\times \mathbb{R}^{n},  \label{Leray-Lions type 2}
\end{equation}%
for a constant $m_{0}>0$. For instance $a_{0}\left( x,u,\xi \right) =f\left(
x,u,\xi \right) \left\vert u\right\vert ^{p-2}u$ is a typical example; here $%
f\left( x,u,\xi \right) $, is a bounded Carath\'{e}odory function far from
zero; i.e. $f\left( x,u,\xi \right) \geq m_{0}>0$. Note that, by (\ref%
{Leray-Lions type 2}), $a_{0}\left( x,u,\xi \right) $ cannot be identically
equal to zero; thus the \textit{Leray-Lions type} differential operator in (%
\ref{Leray-Lions type 1}) in general is different from the differential
operator considered by Leray-Lions in Theorem \ref{Leray-Lions Theorem}.
Under this \textit{Leray-Lions type} operator $A\left( u\right) $,
Boccardo-Murat-Puel obtained the existence of a weak solution to the related
Dirichlet problem with a general right hand side $b=b(x,u,\xi )$ 
\begin{equation*}
\left\{ 
\begin{array}{l}
\sum_{i=1}^{n}\frac{\partial }{\partial x_{i}}a^{i}\left( x,u\left( x\right)
,Du\left( x\right) \right) +a_{0}\left( x,u\left( x\right) ,Du\left(
x\right) \right) =b(x,u,Du),\;x\in \Omega \,, \\ 
u=0\;\;\;\text{on }\partial \Omega \,.%
\end{array}%
\right.
\end{equation*}%
See also the related and more recent articles by\ Ferone-Murat \cite%
{Ferone-Murat 2000}, Alvino-Boccardo-Ferone-Orsina-Trombetti \cite%
{Alvino-Boccardo-Ferone-Orsina-Trombetti 2003}, Abdellaoui-Dall'Aglio-Peral 
\cite{Abdellaoui-Dall'Aglio-Peral 2006}, Arcoya-Carmona-Leonori-Mart\'{\i}%
nezAparicio-Orsina-Petitta \cite%
{Arcoya-Carmona-Leonori-MartinezAparicio-Orsina-Petitta 2009}, Le \cite{Le
2009}, Sanch\'{o}n-Urbano \cite{Sanchon-Urbano 2009},
Alvino-Mercaldo-Volpicelli-Betta \cite{Alvino-Mercaldo-Volpicelli-Betta 2019}
and the book by Boccardo-Croce \cite[Chapter 4]{Boccardo-Croce 2010}.

A case which seems not yet obtained in the literature is the extension of
the Leray-Lions existence theorem to general growth conditions. What do we
mean with \textit{general growth conditions}? Here some simple examples of
differential operators to whom the Leray-Lions theorem in the original
version, with the so called \textit{natural growth conditions}, cannot be
applied:

\begin{equation}
\text{\textit{logaritmic} forms:}\;\;\;\;\;A\left( u,v\right) :=\int_{\Omega
}\sum_{i=1}^{n}\log (1+\left\vert Du\right\vert ^{2})\,\left\vert
Du\right\vert ^{p-2}u_{x_{i}}v_{x_{i}}\,dx\;;  \label{example 1}
\end{equation}%
\begin{equation}
\text{\textit{variable exponents} forms:}\;\;\;\;\;A\left( u,v\right)
:=\int_{\Omega }\sum_{i=1}^{n}\left\vert Du\right\vert ^{p\left( x\right)
-2}u_{x_{i}}v_{x_{i}}\,dx\,;  \label{example 2}
\end{equation}%
\begin{equation}
\text{\textit{anisotropic} forms:}\;\;\;\;\;A\left( u,v\right)
:=\int_{\Omega }\sum_{i=1}^{n}\left\vert u_{x_{i}}\right\vert
^{p_{i}-2}u_{x_{i}}v_{x_{i}}\,dx\,;  \label{example 3}
\end{equation}%
\begin{equation}
\text{\textit{double phase} forms:}\;\;\;\;\;A\left( u,v\right)
:=\int_{\Omega }\sum_{i=1}^{n}(\left\vert Du\right\vert ^{p-2}+a\left(
x\right) \left\vert Du\right\vert ^{q-2})\,u_{x_{i}}v_{x_{i}}\,dx\,.
\label{example 4}
\end{equation}%
Variations with $u$ dependence could be considered too, such as for instance 
$p=p\left( x,u\right) $, $p_{i}=p_{i}\left( x,u\right) $ and $a=a\left(
x,u\right) $. Similarly, relevant examples can be exhibit with \textit{%
nondegenerate energies} and nondegenerate differential operators when above
we interchange $\left\vert Du\right\vert $ with $(1+\left\vert Du\right\vert
^{2})^{1/2}$; thus for instance (\ref{example 2}) corresponds to 
\begin{equation}
\text{\textit{nondegenerate var-exp} forms:}\;\;\;A\left( u,v\right)
:=\int_{\Omega }\sum_{i=1}^{n}(1+\left\vert Du\right\vert ^{2})^{\frac{%
p\left( x\right) -2}{2}}u_{x_{i}}v_{x_{i}}\,dx\,.  \label{example 5}
\end{equation}%
All the previous examples enter in the so called context of $p,q-$\textit{%
conditions}, when we have $p-$\textit{ellipticity} from below and $q-$%
\textit{growth} from above (see the details in the next section). In
particular the logaritmic form in (\ref{example 1}) is elliptic with
exponent $p$ and it satisfies $q-$growth with any exponent $q>p$. The
variable exponent forms in the examples (\ref{example 2}),(\ref{example 5})
are $p-$elliptic with $p=\min \left\{ p\left( x\right) :\;x\in \overline{%
\Omega }\right\} $ and exhibit $q-$growth with $q=\max \left\{ p\left(
x\right) :\;x\in \overline{\Omega }\right\} $. Similarly, also the
anisotropic form in (\ref{example 3}) satisfies a kind of $p,q-$condition
with $p=\min_{i\in \left\{ 1,2,\ldots ,n\right\} }\left\{ p_{i}\right\} $, $%
q=\max_{i\in \left\{ 1,2,\ldots ,n\right\} }\left\{ p_{i}\right\} $. The 
\textit{double phase} form in the example (\ref{example 4}) is characterized
by the coefficient $a\left( x\right) $, which usually is assumed to be
either H\"{o}lder continuous or local Lipschitz continuous in $\Omega $. The
peculiarity is that $a\left( x\right) \geq 0$ in $\Omega $, with the
possibility to be equal to zero on a closed subset of $\Omega $. Therefore
the differential form in (\ref{example 4}) when $p<q$ behaves like a $q-$%
Laplacian in the subset of $\Omega $ where $a\left( x\right) >0$; in this
case the $p-$addendum plays the role of a \textit{``lower order term"}. While
it is a $p-$Laplacian in the subset of $\Omega $ where $a\left( x\right) =0$.

In this research we prove an existence result of weak solutions $%
W_{0}^{1,p}\left( \Omega \right) \cap W_{\mathrm{loc}}^{1,q}\left( \Omega
\right) $, to the Dirichlet problem for a second order elliptic equation in
divergence form (\ref{Dirichlet problem}), under \textit{general} and $p,q-$%
\textit{growth conditions}. Our aim is to extend, as much as possible, to
general growth the Leray-Lions Theorem \ref{Leray-Lions Theorem}. We found a
way to treat the general context with explicit dependence on $\left(
x,u\right) $, other than on the gradient variable $\xi =Du$; these aspects
require particular attention due to the $p,q-$context, with some differences
and new difficulties compared to the standard case $p=q$.

We state in Theorem \ref{Leray-Lions Theorem under p,q-growth} our existence
result, which is proved in Section \ref{Section with the proof}. We make use
of the well-known \textit{nonlinear Schauder approach}, where the principal
tools are the \textit{a-priori estimates}, the \textit{local Lipschitz
continuity} and the $W_{\mathrm{loc}}^{2,2}\left( \Omega \right) -$\textit{%
summability results}, recently obtained by the authors in \cite%
{Cupini-Marcellini-Mascolo regularity 2024} under general $p,q-$conditions
and explicit dependence on $\left( x,u\right) $. We construct a suitable $%
\varepsilon -$sequence of regular equations with standard growth for which
the Leray-Lions Theorem \ref{Leray-Lions Theorem} is valid. By applying to
the approximating solutions $u_{\varepsilon }$ the a-priori estimates, which
turn out to be independent of $\varepsilon $, in the limit we obtain a
solution to the original problem with the regularity properties (\ref%
{gradient estimate}),(\ref{second derivatives estimate}) of Theorem \ref%
{Leray-Lions Theorem under p,q-growth}.

\bigskip

\section{Existence under general $p,q-$growth\label{Section with the
statement}}

We study the Dirichlet problem (\ref{Dirichlet problem}) when the vector
field $a\left( x,u,\xi \right) $ is locally Lipschitz continuous in $\Omega
\times \mathbb{R}\times \mathbb{R}^{n}$, with $\Omega $ bounded open set in $%
\mathbb{R}^{n}$, $n\geq 2$. The right hand side $b\left( x\right) $ is a
measurable function in $\Omega $. We consider two exponents $p,q$, with $%
2\leq p\leq q<p+1$, and the \textit{ellipticity condition} 
\begin{equation}
\sum_{i,j=1}^{n}\tfrac{\partial a^{i}}{\partial \xi _{j}}\lambda _{i}\lambda
_{j}\geq m(1+\left\vert \xi \right\vert ^{2})^{\frac{p-2}{2}}\left\vert
\lambda \right\vert ^{2}\;,  \label{ellipticity}
\end{equation}%
valid for a positive constant $m$ and for every $\lambda ,\xi \in \mathbb{R}%
^{n}$, $\left( x,u\right) \in \Omega \times \mathbb{R}$. We also assume the 
\textit{growth conditions} 
\begin{equation}
\left\vert \tfrac{\partial a^{i}}{\partial \xi _{j}}\right\vert \leq
M(1+\left\vert \xi \right\vert ^{2})^{\frac{q-2}{2}}+M\left\vert
u\right\vert ^{\alpha }\,,\;\;\;\;\;\left\vert \tfrac{\partial a^{i}}{%
\partial u}\right\vert \leq M(1+\left\vert \xi \right\vert ^{2})^{\frac{p+q-4%
}{4}}+M\left\vert u\right\vert ^{\beta -1}\,,  \label{growth 1}
\end{equation}%
with $M>0$ and $0\leq \alpha \leq \frac{2\left( q-2\right) }{q-p+2}$, $0\leq
\beta <p-1$. %$0\leq
%\beta <\frac{n\left( p-1\right) }{n-p}:=\frac{p-1}{p}p^{\ast }$
% \tcr{and 
%\[\beta <\frac{q}{p}(q-1).\]}
Moreover, for every open set $\Omega ^{\prime }$, whose closure is contained
in $\Omega ,$ and for every $L>0$, there exist a positive constant $M\left(
L\right) $ (depending on $\Omega ^{\prime }$ and $L$) such that, for every $%
x\in \Omega ^{\prime }$, $\xi \in \mathbb{R}^{n}$ and for$\;\left\vert
u\right\vert \leq L$,\textit{\ }%
\begin{equation}
\left\vert \tfrac{\partial a^{i}}{\partial \xi _{j}}-\tfrac{\partial a^{j}}{%
\partial \xi _{i}}\right\vert \leq M\left( L\right) (1+\left\vert \xi
\right\vert ^{2})^{\frac{p+q-4}{4}}\,,\;\;\;\;\;\left\vert \tfrac{\partial
a^{i}}{\partial x_{s}}\right\vert \leq M\left( L\right) (1+\left\vert \xi
\right\vert ^{2})^{\frac{p+q-2}{4}},  \label{growth 2}
\end{equation}%
for every $\xi \in \mathbb{R}^{n}$, $i,j,s=1,2,\ldots ,n$. We also assume%
\begin{equation}
\left\{ 
\begin{array}{l}
\left\vert a\left( \cdot ,0,0\right) \right\vert \in L_{\mathrm{loc}%
}^{\gamma ^{\;}}\left( \Omega \right) \,, \\ 
b\in L_{\mathrm{loc}}^{s_{0}}\left( \Omega \right) \,,%
\end{array}%
\right.  \label{growth 3}
\end{equation}%
with $\gamma >\frac{n}{p-1}$ in (\ref{growth 3})$_{1}$ and $s_{0}>n$ in (\ref%
{growth 3})$_{2}$. We assume also a summability condition \textit{valid in
the full }$\Omega $; i.e., 
\begin{equation}
	\left\{ 
	\begin{array}{lc}
		\left\vert a\left( \cdot ,0,0\right) \right\vert \in L^{p^{\prime }\cdot
			q^{\prime }}( \Omega)\,, & p^{\prime }:=\tfrac{p}{p-1}\,, \\ 
		b\in L^{q^{\prime }}( \Omega) \,, & q^{\prime }:=\tfrac{q}{q-1}%
		\,.%
	\end{array}%
	\right.  \label{growth 4}
\end{equation}
The summability assumptions  on $b$ are satisfied if $b\in L^{s_{0}}(\Omega )$%
, $s_{0}>n$; in fact $s_{0}>n\geq 2\geq \frac{q}{q-1}$ since $q\ge p\geq 2$. The summability assumption 
(\ref{growth 4})$_1$ of course  implies that  $	 a\left( \cdot ,0,0\right) $ belongs to $ L^{p^{\prime }}(\Omega)$ and to $L^{q^{\prime }}(\Omega)$.

\medbreak
Under the $p,q-$growth conditions with $p\leq q$, the Sobolev class $W_{%
\mathrm{loc}}^{1,q}\left( \Omega \right) $ is the natural class were to look
for solutions; see \cite[Section 3.1]{Marcellini 2023} for a discussion
about this aspect. Precisely, under the above $p,q-$growth conditions, a 
\textit{weak solution} to the Dirichlet problem (\ref{Dirichlet problem}) is
a function $u\in W_{0}^{1,p}\left( \Omega \right) \cap W_{\mathrm{loc}%
}^{1,q}\left( \Omega \right) $ such that, for every open set $\Omega
^{\prime }\subset \subset \Omega $, 
\begin{equation}
\int_{\Omega }\sum_{i=1}^{n}a^{i}\left( x,u,Du\right)
v_{x_{i}}dx+\int_{\Omega }b\left( x\right) v\,dx=0,\;\;\forall \,v\in
W_{0}^{1,q}\left( \Omega ^{\prime }\right) \,.
\label{definition of weak solution}
\end{equation}

\medbreak
Our main existence  (and regularity)  result is the following.
\begin{theorem}
\label{Leray-Lions Theorem under p,q-growth} Under the ellipticity condition
(\ref{ellipticity}) and the growth conditions (\ref{growth 1})-(\ref{growth
4}), if 
\begin{equation}
\tfrac{q}{p}<1+\tfrac{1}{n}\,,\,  \label{q/p bound}
\end{equation}%
then the Dirichlet problem (\ref{Dirichlet problem}) has (at least) a weak
solution in $W_{0}^{1,p}\left( \Omega \right) \cap W_{\mathrm{loc}%
}^{1,q}\left( \Omega \right) $. Moreover this weak solution has more
interior regularity, in the sense that $u\in W_{\mathrm{loc}}^{1,\infty
}\left( \Omega \right) \cap W_{\mathrm{loc}}^{2,2}(\Omega )$ and the
following gradient and $W^{2,2}$ local estimates hold. In particular $u\in
L_{\mathrm{loc}}^{\infty }\left( \Omega \right) $ and, fixed $\Omega
^{\prime }\subset \subset \Omega $, the constant $c$ below depends on the $%
L_{\mathrm{loc}}^{\infty }\left( \Omega ^{\prime }\right) $ bound of $u$ but
not on $u$ itself: for the gradient there exists a parameters $\alpha \geq 1$
such that 
\begin{equation}
\left\Vert Du\left( x\right) \right\Vert _{L^{\infty }\left( B_{\varrho
}\right) }\leq \left( \tfrac{c}{\left( R-\varrho \right) ^{n}}%
\int_{B_{R}}(1+\left\vert Du\left( x\right) \right\vert ^{2})^{\frac{p}{2}%
}\,dx\right) ^{\frac{\alpha }{p}}  \label{gradient estimate}
\end{equation}%
\begin{equation*}
\underset{\text{for }n>2}{=}\;\left( \tfrac{c}{\left( R-\varrho \right) ^{n}}%
\left\Vert (1+\left\vert Du\left( x\right) \right\vert ^{2})^{\frac{1}{2}%
}\right\Vert _{L^{p}\left( B_{R}\right) }^{p}\right) ^{\frac{2}{\left(
n+2\right) p-nq}}\,,
\end{equation*}%
and for the $n\times n$ matrix $D^{2}u$ of the second derivatives 
\begin{equation}
\int_{B_{\rho }}\left\vert D^{2}u\right\vert ^{2}\,dx\leq \tfrac{c}{\left(
R-\rho \right) ^{2}}\int_{B_{R}}(1+\left\vert Du\left( x\right) \right\vert
^{2})^{\frac{q}{2}}\,dx\,  \label{second derivatives estimate}
\end{equation}%
\begin{equation*}
\leq \frac{c^{\prime }}{\left( R-\rho \right) ^{2}}\left( \tfrac{1}{\left(
R-\varrho \right) ^{\gamma \vartheta \left( q-p\right) }}\int_{B_{R}}(1+%
\left\vert Du\left( x\right) \right\vert ^{2})^{\frac{p}{2}}\,dx\right) ^{%
\frac{\alpha q}{\vartheta p}}\,.
\end{equation*}
\end{theorem}

The proof is given in the next section.

\begin{remark}
\label{Remark after the existence theorem}We already noted in the
introduction that, following the original existence theorem by Leray-Lions,
in the literature some similar ``\textit{Leray-Lions type} operators" exist;
such as, for instance, the differential operator $A\left( u\right) $ defined
in (\ref{Leray-Lions type 1}) under the further coercivity assumption (\ref%
{Leray-Lions type 2}). In Theorem \ref{Leray-Lions Theorem under p,q-growth}
we adopt and generalize the original Leray-Lions approach in \cite%
{Leray-Lions 1965}, as explicitly described in Theorem \ref{Leray-Lions
Theorem} above. Nevertheless it is of interest to find and prove existence
results for other similar \textit{Leray-Lions type} differential operators
under general growth conditions as those described in (\ref{growth 1})-(\ref%
{growth 4}), with the ellipticity as in (\ref{ellipticity}). In particular
it would be of interest to consider right hand sides $b=b\left(
x,u,Du\right) $ explicitly depending on $\left( u,Du\right) $ too; with
respect to this aspect we mention that the regularity theory that we use in
the proof of the following Step 8, i.e. the local gradient estimate (\ref%
{gradient bound in the proof}) and the local bound on the second derivatives
(\ref{bound on second derivatives in the proof}), hold for a general right
hand side $b=b\left( x,u,Du\right) $ too (see \cite%
{Cupini-Marcellini-Mascolo regularity 2024}).
\end{remark}

We emphasize that Theorem \ref{Leray-Lions Theorem under p,q-growth} proves
existence of at least one solution to the Dirichlet problem (\ref{Dirichlet
problem}) under general ellipticity and growth conditions, in the spirit of
the Leray-Lions Theorem \ref{Leray-Lions Theorem}. As well Theorem \ref%
{Leray-Lions Theorem}, the result of Theorem \ref{Leray-Lions Theorem under
p,q-growth} does not gives \textit{multiplicity} of weak solutions. However
the researches in these fields with general growth conditions, \textit{%
logaritmic}, \textit{variable exponents, anisotropic, double phase},
nowadays is so wide that it is correct to give some references of some
multiplicity results obtained for specific elliptic - sometime also
parabolic - Dirichlet problems. It is not possible to give a full landscape
of all the papers\ related to multiplicity and to existence of nontrivial
solution; however we quote at least Mih\u{a}ilescu-Pucci-R\u{a}dulescu \cite%
{Mihailescu-Pucci-Radulescu 2008}, Papageorgiou-R\u{a}dulescu-Zhang \cite%
{Papageorgiou-Radulescu-Zhang 2022}, Fang-R\u{a}dulescu-Chao Zhang-Xia Zhang 
\textsc{\cite{Fang-Radulescu-Zhang-Zhang 2022},} Liu Jingjing-Patrizia Pucci 
\cite{Liu-Pucci 2023} and the references therein.

Nowadays the literature on $p,q-$problems is large, mainly devoted to
regularity. The first approach can be found in the $90^{\prime }s$ in \cite%
{Marcellini ARMA 1989},\cite{Marcellini 1991}, not only in \textit{the }$%
p,q- $\textit{context} but also for \textit{general nonuniformly elliptic
problems}, see more recently \cite{DiMarco-Marcellini 2020},\cite{Marcellini
2020}. Recently a strong impulse, with the introduction of the terminology
of \textit{double phase integrals} and fine results, was given by Baroni,
Colombo and Mingione in \cite{Colombo-Mingione 2015},\cite%
{Baroni-Colombo-Mingione 2018}. \textit{Variable exponents} and double phase
problems are considered by Byun-Oh \cite{Byun-Oh 2020}, Ragusa-Tachikawa 
\cite{Ragusa-Tachikawa 2020}. For \textit{Orlicz-Sobolev spaces} see
Diening-Harjulehto-Hasto-Ruzicka \cite{Diening-Harjulehto-Hasto-Ruzicka 2011}%
, Chlebicka \cite{Chlebicka 2018}, Chlebicka-DeFilippis \cite%
{Chlebicka-DeFilippis 2019}, H\"{a}st\"{o}-Ok \cite{Hasto-Ok 2022}. About%
\textit{\ quasiconvex integrals} of the calculus of variations see in
particular \cite{Boegelein-Dacorogna-Duzaar-Marcellini-Scheven 2020},\cite%
{Dacorogna-Marcellini 1998},\cite{Marcellini 1984} and
DeFilippis-Stroffolini \cite{De Filippis JMPA 2022},\cite{De
Filippis-Stroffolini 2023}. Most of the quoted results deal with \textit{%
interior} regularity, such as \cite{Cupini-Marcellini-Mascolo 2017},\cite%
{Cupini-Marcellini-Mascolo 2018},\cite{Cupini-Marcellini-Mascolo-Passarelli
2023},\cite{Eleuteri-Marcellini-Mascolo-Perrotta 2022},\cite%
{Eleuteri-Passarelli 2023},\cite{Fang-Radulescu-Zhang-Zhang 2022}, apart
from Cianchi-Maz'ya \cite{Cianchi-Mazya 2011},\cite{Cianchi-Mazya 2014}, B%
\"{o}gelein-Duzaar-Marcellini-Scheven \cite%
{Boegelein-Duzaar-Marcellini-Scheven JMPA 2021}, DeFilippis-Piccinini \cite%
{Defilippis-Piccinini 2022-2023}.

\section{Proof of Theorem \protect\ref{Leray-Lions Theorem under p,q-growth} 
\label{Section with the proof}}

With the aim to obtain the existence of a weak solution to the Dirichlet
problem (\ref{Dirichlet problem}) we start by assumption (\ref{q/p bound})
and we consider a parameter $\varepsilon \in \left( 0,\varepsilon _{0}\right]
$, were $\varepsilon _{0}$ is a fixed positive real number such that 
\begin{equation}
\tfrac{q+\varepsilon _{0}}{p}<1+\tfrac{1}{n}\,.
\label{q/p bound in the proof}
\end{equation}%
For every $\varepsilon \in \left( 0,\varepsilon _{0}\right] $ we introduce
the \textit{approximated Dirichlet problem} 
\begin{equation}
\left\{ 
\begin{array}{l}
\sum_{i=1}^{n}\frac{\partial }{\partial x_{i}}\left\{ a^{i}\left( x,u\left(
x\right) ,Du\left( x\right) \right) +\varepsilon \,\left( 1+\left\vert
Du\right\vert ^{2}\right) ^{\frac{q+\varepsilon -2}{2}}u_{x_{i}}\right\}
=b\left( x\right) ,\;\;x\in \Omega \,, \\ 
u=0\;\;\;\text{on }\partial \Omega \,.%
\end{array}%
\right.  \label{approximated Dirichlet problem}
\end{equation}

We start describing the scheme of the proof of Theorem \ref{Leray-Lions
Theorem under p,q-growth}. The first part consists in proving that, for
every $\varepsilon $, there exists a weak solution $u_{\varepsilon }\in
W_{0}^{1,q+\varepsilon }\left( \Omega \right) $ to the approximated
Dirichlet problem (\ref{approximated Dirichlet problem}); at this stage the
Leray-Lions Theorem \ref{Leray-Lions Theorem} pays a role, see Steps 1-4
below. Two types of Sobolev-estimates - \textit{local} and \textit{global}
ones - have a central role in the proof of Theorem \ref{Leray-Lions Theorem
under p,q-growth}: \textit{(i)} the \textit{interior} $L_{\mathrm{loc}%
}^{\infty }\left( \Omega \right) -$gradient estimates for $Du$ and the $L_{%
\mathrm{loc}}^{2}\left( \Omega \right) -$estimate for the $n\times n$ matrix 
$D^{2}u$ of the second derivatives, recently obtained by the authors in \cite%
{Cupini-Marcellini-Mascolo regularity 2024}; \textit{(ii)} the \textit{global%
} $L^{p}\left( \Omega \right) -$gradient estimate proved in the following
Lemma \ref{global gradient estimate}. These interior and global regularity
properties will be uniform with respect to $\varepsilon \in \left(
0,\varepsilon _{0}\right] $ and will allow us to go to the limit as $%
\varepsilon $ goes to zero, obtaining the desired solution to Dirichlet
problem (\ref{Dirichlet problem}).

\begin{lemma}[Global $L^{p}\left( \Omega \right) -$gradient estimate]
\label{global gradient estimate} There exists a constant $c$, depending on $%
n,p,\beta $ and the $n$-Lebesgue measure of $\Omega $, but not on $%
\varepsilon $, such that 
\begin{equation}
\Vert Du_{\varepsilon }\Vert _{L^{p}(\Omega )}^{p}\leq \,c\Big(1+\Vert
a(\cdot ,0,0)\Vert _{L^{\frac{p}{p-1}}(\Omega )}+\Vert b\Vert _{L^{\frac{%
p^{\ast }}{p^{\ast }-1}}(\Omega )}\Big)^{\frac{p}{p-1}}.
\label{e:Lpuniforme}
\end{equation}
\end{lemma}

See the proof of Lemma \ref{global gradient estimate} in the Step 6 below.
Recall that, as usual, $p^{\ast }:=\frac{np}{n-p}$ if $p<n$, otherwise $%
p^{\ast }$ is any real number greater than $p$; since we can choose its
value, we fix $p^{\ast }>q$ also when $p\geq n$ (and we observe that, if $%
p<n $, then by (\ref{q/p bound}) automatically $p^{\ast }>q$). Notice that
by (\ref{growth 4}) the right hand side in (\ref{e:Lpuniforme}) is finite
and bounded independently of $\varepsilon $. In fact  $p^*>q $ implies  $L^{q'}(\Omega)\subseteq 
L^{(p^*)'}(\Omega)$, therefore the $L^{(p^*)'}-$norm of $b$ is finite due to   (\ref{growth 4})$_2$.  Moreover the $L^{p'}-$ norm of $a(\cdot,0,0)$ is finite because of (\ref{growth 4})$_1$, since $q'>1$.

%Moreover the $L^{p'}-$ norm of $a(\cdot,0,0)$ is finite because of (\ref{growth 4})$_1$, since $p'< (q')^2$. This last inequality  explicitly means  $q^2-2qp+p< 0$, which is equivalent to $q<p+\sqrt{p^2-p}$, which is implied by the bound (\ref{q/p bound}):  $q<p+\frac{p}{n}<p+\sqrt{p^2-p}$ valid for every $p\ge 2$ and $n\ge 2$.   

\medbreak
In the proof of Theorem \ref{Leray-Lions Theorem under p,q-growth} we use
the following Lemma too, which is a variant of the author's Linking Lemma
3.1-\textit{(i)} in \cite{Cupini-Marcellini-Mascolo regularity 2024}.

\begin{lemma}
\label{variant of the Linking Lemma (i)}Under the ellipticity condition (\ref%
{ellipticity}) and the order-one growth condition (\ref{growth 1})$_{2}$
with $0\leq \beta \leq p-1$, the following bound from below holds 
\begin{equation}
\left( a\left( x,u,\xi \right) ,\xi \right) \geq -c\left( \left\vert \xi
\right\vert ^{q}+\left\vert u\right\vert ^{q}+\left\vert a\left(
x,0,0\right) \right\vert ^{\frac{q}{q-1}}+1\right) \,,
\label{(i) in the Lemma}
\end{equation}%
for a positive constant $c$ and for all $x\in \Omega $, $u\in \mathbb{R}$
and $\xi \in \mathbb{R}^{n}$.
\end{lemma}

\begin{proof}
With the component notation $\xi =\left( \xi _{i}\right) _{i=1,\ldots ,n}$,
we have 
\begin{equation*}
\left( a\left( x,u,\xi \right) ,\xi \right) -\left( a\left( x,0,0\right)
,\xi \right)
\end{equation*}%
\begin{equation*}
=\sum_{i=1}^{n}\left\{ a^{i}\left( x,u,\xi \right) -a^{i}\left( x,0,0\right)
\right\} \xi _{i}=\int_{0}^{1}\frac{d}{dt}\sum_{i=1}^{n}a^{i}\left(
x,tu,t\xi \right) \xi _{i}\,dt
\end{equation*}%
\begin{equation*}
=\int_{0}^{1}\left\{ \sum_{i=1}^{n}a_{u}^{i}\left( x,tu,t\xi \right) \xi
_{i}u+\sum_{i,j=1}^{n}a_{\xi _{j}}^{i}\left( x,tu,t\xi \right) \xi _{i}\xi
_{j}\right\} \,dt\,.
\end{equation*}%
We use the growth condition (\ref{growth 1})$_{2}$ and the ellipticity
assumption (\ref{ellipticity}); we get 
\begin{equation*}
\left( a\left( x,u,\xi \right) ,\xi \right) -\left( a\left( x,0,0\right)
,\xi \right)
\end{equation*}%
\begin{equation*}
\geq \int_{0}^{1}\left\{ -nM\left( (1+\left\vert t\xi \right\vert ^{2})^{%
\frac{p+q-4}{4}}\left\vert \xi \right\vert \left\vert u\right\vert
+\left\vert \xi \right\vert \left\vert u\right\vert ^{\beta }\right)
+m\left( 1+\left\vert t\xi \right\vert ^{2}\right) ^{\frac{p-2}{2}%
}\left\vert \xi \right\vert ^{2}\right\} dt
\end{equation*}%
\begin{equation*}
\ge \int_{0}^{1}  -nM\left( (1+\left\vert t\xi \right\vert ^{2})^{\frac{%
p-2}{4}}\left\vert \xi \right\vert (1+\left\vert t\xi \right\vert ^{2})^{%
\frac{q-2}{4}}\left\vert u\right\vert +\left\vert \xi \right\vert \left\vert
u\right\vert ^{\beta }\right)  \,dt 
\end{equation*}%
\begin{equation*}
\geq \int_{0}^{1}  -nM\left( (1+\left\vert t\xi \right\vert ^{2})^{%
\frac{p-2}{2}}\left\vert \xi \right\vert ^{2}+\tfrac{1}{4}(1+\left\vert t\xi
\right\vert ^{2})^{\frac{q-2}{2}}\left\vert u\right\vert ^{2}+\left\vert \xi
\right\vert \left\vert u\right\vert ^{\beta }\right)  \,dt
\end{equation*}%
%Therefore, there exists a positive constant $c$ such that 
%\begin{equation*}
%\left( a\left( x,u,\xi \right) ,\xi \right) -\left( a\left( x,0,0\right)
%,\xi \right)
%\end{equation*}%
%\begin{equation*}
%\geq \int_{0}^{1}\left\{ -c(1+\left\vert t\xi \right\vert ^{2})^{\frac{q-2}{2%
%}}\left\vert u\right\vert ^{2}-nM\left\vert \xi \right\vert \left\vert
%u\right\vert ^{\beta }-nM(1+\left\vert t\xi \right\vert ^{2})^{%
%\frac{p-2}{2}}\left\vert \xi \right\vert ^{2}\right\} dt.
%\end{equation*}%
We take $t=1$ in the integrands and we apply Young's inequality with
conjugate exponents $\frac{q}{q-2}$ and $\frac{q}{2}$, while in the last 
addendum we consider the conjugate exponents $p$ and $\frac{p}{p-1}$%
\begin{equation*}
\left( a\left( x,u,\xi \right) ,\xi \right) -\left( a\left( x,0,0\right)
,\xi \right) \geq -c
(1+|\xi|^2)^{\frac{p-2}{2}}|\xi|^2 -c (1+\left\vert \xi \right\vert ^{2})^{\frac{%
q}{2}}
\end{equation*}%
\begin{equation*}
-c \left\vert u\right\vert ^{q}-c\left\vert \xi \right\vert ^{p}-c\left\vert u\right\vert ^{\beta \frac{p}{p-1}}.
\end{equation*}%
By Young's inequality with conjugate exponents $\frac{q}{q-1}$ and $q$ 
\begin{equation}
	\left\vert \left( a\left( x,0,0\right) ,\xi \right) \right\vert \leq
	\left\vert a\left( x,0,0\right) \right\vert \left\vert \xi \right\vert \leq 
	\tfrac{q-1}{q}\left\vert a\left( x,0,0\right) \right\vert ^{\frac{q}{q-1}}+%
	\tfrac{1}{q}\left\vert \xi \right\vert ^{q}\,.  \label{consequences 0-a}
\end{equation}%
We arrive to a bound from below of the type 
\begin{equation*}
\left( a\left( x,u,\xi \right) ,\xi \right) \geq - c\left(  
(1+|\xi|^2)^{\frac{p}{2}}+ 
(1+|\xi|^2)^{\frac{q}{2}}
+\left\vert u\right\vert ^{\beta \frac{p}{p-1}}  +\left\vert u\right\vert
^{q}\right) -\left\vert a\left( x,0,0\right) \right\vert ^{\frac{q}{q-1}}\,;
\end{equation*}%
this bound from below implies the conclusion (\ref{(i) in the Lemma}), since 
$\beta \frac{p}{p-1}\leq p\leq q$,  $\left\vert u\right\vert ^{\beta \frac{%
p}{p-1}}\leq 1+\left\vert u\right\vert ^{q}$ 
and 
$(1+|\xi|^2)^{\frac{p}{2}}\le 
(1+|\xi|^2)^{\frac{q}{2}}\le c(1+|\xi|^q)$.
\end{proof}

\bigskip

Let us first verify that the assumptions in the Leray-Lions Theorem \ref%
{Leray-Lions Theorem} are satisfied with respect to the $\left(
q+\varepsilon \right) -$exponent and to the differential operator associated
to the vector field $a_{\varepsilon }\left( x,u,\xi \right) =\left(
a_{\varepsilon }^{i}\left( x,u,\xi \right) \right) _{i=1,\ldots ,n}$, $%
\varepsilon \in \left( 0,\varepsilon _{0}\right] $, given by 
\begin{equation}
a_{\varepsilon }^{i}\left( x,u,\xi \right) :=a^{i}\left( x,u,\xi \right)
+\varepsilon \,(1+\left\vert \xi \right\vert ^{2})^{\frac{q+\varepsilon -2}{2%
}}\xi _{i}\,.  \label{definition of a-epsilon}
\end{equation}%
The right hand side $b\left( x\right) $ in (\ref{approximated Dirichlet
problem}) satisfies the needed summability condition. In fact the assumption 
$b\in L_{\mathrm{loc}}^{s_{0}}\left( \Omega \right) $ in (\ref{growth 3})$%
_{2}$ with $s_{0}>n$ and the fact that $q\geq 2$ implies that $2\geq \frac{q%
}{q-1}$ and so $s_{0}>n\geq 2\geq \frac{q}{q-1}=q^{\prime }$; therefore $%
b\in L_{\mathrm{loc}}^{q^{\prime }}\left( \Omega \right) $ too.

\medbreak\textit{Step 1:} The condition (\ref{Leray-Lions 1}), in the
context considered here, takes the form 
\begin{equation}
\left\vert a_{\varepsilon }\left( x,u,\xi \right) \right\vert \leq M\left(
\left\vert \xi \right\vert ^{q+\varepsilon -1}+\left\vert u\right\vert
^{q +\varepsilon -1}+b_{1}\left( x\right) \right)  \label{e:crescitaaepsGC}
\end{equation}%
with 
\begin{equation}
b_{1}\left( x\right) :=  1+\left\vert a\left( x,0,0\right)
\right\vert^{\frac{p}{p-1}}.  \label{consequences 0-b}
\end{equation}
%The growth condition (\ref{e:crescitaaepsGC}) it is a consequence of the
%Linking Lemma 3.1-\textit{(ii)} in \cite{Cupini-Marcellini-Mascolo
%regularity 2024}, with $q$ replaced by $q+\varepsilon $: by (\ref{growth 1})
%and by the fact that $\varepsilon \leq \varepsilon _{0}$, 
%\begin{equation*}
%\left\vert a_{\varepsilon }\left( x,u,\xi \right) \right\vert \leq
%c_{3}\left\vert \xi \right\vert ^{q+\varepsilon -1}+c_{4}\left\vert
%u\right\vert ^{2\frac{q+\varepsilon -1}{q+\varepsilon -p+2}}+b_{1}\left(
%x\right).
%\end{equation*}
%Since $q+\varepsilon >q\geq p$ we also have $2\frac{q+\varepsilon -1}{%
%q+\varepsilon -p+2}\leq 2\frac{q+\varepsilon -1}{2}=q+\varepsilon -1$ and we
%obtain $\left\vert a_{\varepsilon }\left( x,u,\xi \right) \right\vert \leq
%M(\left\vert \xi \right\vert ^{q+\varepsilon -1}+\left\vert u\right\vert
%^{q+\varepsilon -1}+b_{1}\left( x\right) )$, as required in (\ref%
%{Leray-Lions 1}) with $p$ replaced by $q+\varepsilon $.
%

The growth condition (\ref{e:crescitaaepsGC}) it is a consequence of the
Linking Lemma 3.1-\textit{(ii)} in \cite{Cupini-Marcellini-Mascolo
	regularity 2024}: by (\ref{growth 1}) there exists a constant $c$ such that 
\begin{equation*}
	\left\vert a\left( x,u,\xi \right) \right\vert \leq
	c\, \left(\left\vert \xi \right\vert ^{q -1}+ \left\vert
	u\right\vert ^{2\frac{q  -1}{q  -p+2}}+b_{1}\left(
	x\right)\right)
\end{equation*}
and thus for every  $\varepsilon \leq \varepsilon _{0}$, 
\begin{equation*}
	\left\vert a_{\varepsilon }\left( x,u,\xi \right) \right\vert \leq \varepsilon
	(1+\left\vert \xi \right\vert ^{2})^{\frac{q+\varepsilon -2}{2%
	}}|\xi|
+	c\, \left(\left\vert \xi \right\vert ^{q -1}+ \left\vert
u\right\vert ^{2\frac{q  -1}{q  -p+2}}+b_{1}\left(
x\right)\right).
\end{equation*}
Since $q\geq p$, then  $2\frac{q  -1}{%
	q  -p+2}\leq  q-1$ and we
obtain $\left\vert a_{\varepsilon }\left( x,u,\xi \right) \right\vert \leq
M(\left\vert \xi \right\vert ^{q+\varepsilon -1}+\left\vert u\right\vert
^{q  -1}+b_{1}\left( x\right)+1 )
\leq
M(\left\vert \xi \right\vert ^{q+\varepsilon -1}+\left\vert u\right\vert
^{q +\varepsilon -1}+b_{1}\left( x\right)+2 )
$. Since $b_1\ge 1$ this inequality implies 
\eqref{e:crescitaaepsGC}, possibly with a greater constant $M$.

, that implies

as required in (\ref%
{Leray-Lions 1}) with $p$ replaced by $q+\varepsilon $.

\medbreak\textit{Step 2:} We need to prove (\ref{Leray-Lions 2}) in the form 
\begin{equation}
\lim_{\left\vert \xi \right\vert \rightarrow +\infty }\frac{\left(
a_{\varepsilon }\left( x,u,\xi \right) ,\xi \right) }{\left\vert \xi
\right\vert +\left\vert \xi \right\vert ^{q+\varepsilon -1}}=+\infty \,
\label{Leray-Lions 2 in the proof}
\end{equation}%
for a.e. $x\in \Omega $ and $\left\vert u\right\vert $ bounded. By the
Linking Lemma 3.1-\textit{(i)} in \cite{Cupini-Marcellini-Mascolo regularity
2024}, as consequence of the assumptions (\ref{growth 1}), we obtain 
\begin{equation}
\left( a\left( x,u,\xi \right) ,\xi \right) \geq c_{1}\left\vert \xi
\right\vert ^{p}-c_{2}\left\vert u\right\vert ^{\theta }-b_{1}\left(
x\right) \,,  \label{from the Linking Lemma 3.1-(i)}
\end{equation}%
with $\theta :=\max \left\{ \frac{2p}{p-q +2};\beta \frac{p}{p-1}\right\} $
and $b_{1}\left( x\right) $ as before. Notice that, by (\ref{q/p bound})
then, in particular, $\frac{q }{p}<1+\frac{2}{n}$ which is equivalent to $%
\frac{2p}{p-q +2}<p^{\ast }$; moreover, as in (\ref{growth 1}), $0\leq \beta
<p-1$ and thus $\beta \frac{p}{p-1}<p$; therefore $\theta <p^{\ast }$. The $%
\varepsilon -$term in (\ref{definition of a-epsilon}) gives the contribution 
\begin{equation}
\left( \varepsilon \,(1+\left\vert \xi \right\vert ^{2})^{\frac{%
q+\varepsilon -2}{2}}\xi ,\xi \right) =\varepsilon \,(1+\left\vert \xi
\right\vert ^{2})^{\frac{q+\varepsilon -2}{2}}\left\vert \xi \right\vert
^{2}\geq \varepsilon \,\left\vert \xi \right\vert ^{q+\varepsilon }\,.
\label{in the proof 3}
\end{equation}%
Therefore, since $x$ is fixed in $\Omega $ and $\left\vert u\right\vert $ is
bounded, for every fixed positive $\varepsilon $ the limit condition in (\ref%
{Leray-Lions 2 in the proof}) follows by (\ref{from the Linking Lemma
3.1-(i)}),(\ref{in the proof 3}).

\medskip

\textit{Step 3:} Reasoning as in \cite[Lemma 4.4]{Marcellini 1991}, we can
prove that 
\begin{equation}
\left( a\left( x,u,\xi \right) -a\left( x,u,\eta \right) ,\xi -\eta \right)
\geq m\left( 1+\left\vert \tfrac{\xi +\eta }{2}\right\vert ^{2}\right) ^{%
\frac{p-2}{2}}|\xi -\eta |^{2}\,,  \label{in the proof 4}
\end{equation}%
for every $\xi ,\eta \in \mathbb{R}^{n}$. Indeed, 
\begin{equation*}
\left( a\left( x,u,\xi \right) -a\left( x,u,\eta \right) ,\xi -\eta \right)
=\int_{0}^{1}\sum_{i=1}^{n}\frac{d}{dt}a^{i}\left( x,u,\eta +t(\xi -\eta
)\right) \{\xi _{i}-\eta _{i}\}\,dt
\end{equation*}%
\begin{equation*}
=\int_{0}^{1}\sum_{i,j=1}^{n}a_{\xi _{j}}^{i}\left( x,tu,\eta +t(\xi -\eta
)\right) \{\xi _{i}-\eta _{i}\}\{\xi _{j}-\eta _{j}\}\,dt\,.
\end{equation*}%
By the ellipticity assumption (\ref{ellipticity}), for every $\xi ,\eta \in 
\mathbb{R}^{n}$ we have 
\begin{equation*}
\left( a\left( x,u,\xi \right) -a\left( x,u,\eta \right) ,\xi -\eta \right)
\geq m\int_{0}^{1}(1+\left\vert \eta +t(\xi -\eta )\right\vert ^{2})^{\frac{%
p-2}{2}}\,dt\left\vert \xi -\eta \right\vert ^{2}.
\end{equation*}%
Therefore (\ref{in the proof 4}) holds, because, as shown in \cite%
{Marcellini 1991}, 
\begin{equation*}
\int_{0}^{1}(1+\left\vert \eta +t(\xi -\eta )\right\vert ^{2})^{\frac{p-2}{2}%
}\,dt\geq \left( 1+\left\vert \tfrac{\xi +\eta }{2}\right\vert ^{2}\right) ^{%
\frac{p-2}{2}}.
\end{equation*}%
Since 
\begin{equation*}
\sum_{i=1}^{n}(1+\left\vert \xi \right\vert ^{2})^{\frac{q+\varepsilon -2}{2}%
}(\xi _{i}-\eta _{i})^{2}>0\,,\quad \forall\; \xi ,\eta \in \mathbb{R}^{n}:\;\xi \neq \eta \,,
\end{equation*}%
we obtain 
\begin{equation}
\left( a_{\varepsilon }\left( x,u,\xi \right) -a_{\varepsilon }\left(
x,u,\eta \right) ,\xi -\eta \right) >0,\;\;\;\;\;\forall \;\xi ,\eta \in 
\mathbb{R}^{n}:\;\xi \neq \eta \,  \label{in the proof 5}
\end{equation}%
and condition (\ref{Leray-Lions 3}) in the Leray-Lions Theorem \ref%
{Leray-Lions Theorem} also holds.

\medbreak

\textit{Step 4:} We plan to verify (\ref{Leray-Lions 6}) in our context;
precisely 
\begin{equation}
\lim_{\left\Vert u\right\Vert _{W_{0}^{1,q+\varepsilon }\left( \Omega
\right) }\rightarrow +\infty }\frac{\left\vert A_{\varepsilon }\left(
u,u\right) \right\vert }{\left( \int_{\Omega }\left\vert Du\right\vert
^{q+\varepsilon }\,dx\right) ^{1/\left( q+\varepsilon \right) }}=+\infty \,,
\label{in the proof 6}
\end{equation}%
where $A_{\varepsilon }\left( u,v\right) $ is the form, linear in $v\in
W_{0}^{1, q+\varepsilon }\left( \Omega \right) $, 
\begin{equation*}
A_{\varepsilon }\left( u,v\right) :=\int_{\Omega }\sum_{i=1}^{n}a^{i}\left(
x,u,Du\right) v_{x_{i}}\,dx+\varepsilon \int_{\Omega }\sum_{i=1}^{n}\left(
1+\left\vert Du\right\vert ^{2}\right) ^{\frac{q+\varepsilon -2}{2}%
}u_{x_{i}}v_{x_{i}}\,dx\,.
\end{equation*}%
To prove (\ref{in the proof 6}) we start by Lemma \ref{variant of the
Linking Lemma (i)}. By (\ref{(i) in the Lemma}) and (\ref{in the proof 3})
we obtain 
\begin{equation}
A_{\varepsilon }\left( u,u\right) \geq -c\int_{\Omega }\left( \left\vert
Du\right\vert ^{q}+\left\vert u\right\vert ^{q}+b_{1}\left( x\right) \right)
\,dx+\varepsilon \int_{\Omega }|Du|^{q+\varepsilon }\,dx\,.
\label{in the proof 7}
\end{equation}%
As well known, by H\"{o}lder inequality and by Sobolev embedding theorem we
have 
\begin{equation*}
\left( \int_{\Omega }|u|^{q}\,dx\right) ^{\frac{1}{q}}\leq \left(
\int_{\Omega }\left\vert u\right\vert ^{q^{\ast }}\,dx\right) ^{\frac{1}{%
q^{\ast }}}\left\vert \Omega \right\vert ^{1-\frac{1}{q^{\ast }}}
\end{equation*}%
\begin{equation}
\leq c_{S}\left( \int_{\Omega }\left\vert Du\right\vert ^{q}\,dx\right) ^{%
\frac{1}{q}}\left\vert \Omega \right\vert ^{1-\frac{1}{q^{\ast }}},
\label{in the proof 8}
\end{equation}%
where $c_{S}$ is the Sobolev constant. Moreover, again by H\"{o}lder
inequality, 
\begin{equation}
\left( \int_{\Omega }|Du|^{q}\,dx\right) ^{\frac{1}{q}}\leq \left(
\int_{\Omega }|Du|^{q+\varepsilon }\,dx\right) ^{\frac{1}{q+\varepsilon }%
}\left\vert \Omega \right\vert ^{1-\frac{1}{q+\varepsilon }}\,.
\label{in the proof 9}
\end{equation}%
From (\ref{in the proof 7}),(\ref{in the proof 8}),(\ref{in the proof 9}) we
get 
\begin{equation*}
A_{\varepsilon }\left( u,u\right) \geq -\left( c+c_{S}^q\left\vert \Omega
\right\vert ^{q-\frac{q}{q^{\ast }}}\right) \int_{\Omega }\left\vert
Du\right\vert ^{q}\,dx-c\int_{\Omega }b_{1}\left( x\right) \,dx+\varepsilon
\int_{\Omega }|Du|^{q+\varepsilon }\,dx
\end{equation*}%
\begin{equation*}
\geq -\left( c+c_{S}^q\left\vert \Omega
\right\vert ^{q-\frac{q}{q^{\ast }}}\right) \left\vert \Omega \right\vert ^{q-\frac{q}{q+\varepsilon }}\left(
\int_{\Omega }|Du|^{q+\varepsilon }\,dx\right) ^{\frac{q}{q+\varepsilon }}
\end{equation*}%
\begin{equation*}
-c\int_{\Omega }b_{1}\left( x\right) \,dx+\varepsilon \int_{\Omega
}|Du|^{q+\varepsilon }\,dx\,.
\end{equation*}%
Recall that, at this step, $\varepsilon $ is a positive fixed parameter.
Since $\frac{q}{q+\varepsilon }<1$, by this estimate we see that the limit
as $\left\Vert u\right\Vert _{W_{0}^{1,q+\varepsilon }\left( \Omega \right)
}^{q+\varepsilon }=\int_{\Omega }\left\vert Du\right\vert ^{q+\varepsilon
}\,dx\rightarrow +\infty $ of the form $A_{\varepsilon }\left( u,u\right) $
is equal to $+\infty $, as well as the limit in (\ref{in the proof 6}).

\medbreak

\textit{Step 5:} In this step we discuss the summability properties of the
function $b$ in the Dirichlet problem (\ref{Dirichlet problem}) and of the
function $b_1$ in (\ref{e:crescitaaepsGC}). We are assuming, see (\ref%
{growth 4})$_2$, that $b\in L^{q^{\prime}}(\Omega)$; since $q<q+\varepsilon$%
, then $(q+\varepsilon)^{\prime}\le q^{\prime}$. This implies, in
particular, that $b\in L^{(q+\varepsilon)^{\prime}}(\Omega)$. Analogously,
by (\ref{growth 4})$_1$ we have the global summability assumption $b_1\in
L^{q^{\prime}}\left( \Omega \right)$. Also in this case, we get $b_1\in
L^{(q+\varepsilon)^{\prime}}\left( \Omega \right)$.

%{\color{blue}The function $b_1$ appearing in (\ref{e:crescitaaepsGC}) is defined in  
%	(\ref{consequences 0-b}). 
%	By assumption (\ref{growth 3})$_1$, we have the local summability condition  $b_1\in L_{\mathrm{loc}}^{\gamma \frac{p-1}{p}}(\Omega)$, with  $\gamma >\frac{n}{p-1}$; by 
%In particular, if $p\geq \le n$
%we have  $\gamma >\frac{p}{p-1}=p^{\prime }$; otherwise it is sufficient to take $\gamma >\max \left\{ \frac{n}{p-1};\frac{%
%p}{p-1}\right\} $. 

\begin{remark}
Since (\ref{e:crescitaaepsGC}), (\ref{Leray-Lions 2 in the proof}) and (\ref%
{in the proof 6}) hold true, and since $b, b_1\in
L^{(q+\varepsilon)^{\prime}}\left( \Omega \right)$, then all the assumptions
of Theorem \ref{Leray-Lions Theorem} hold. Therefore, for every $\varepsilon 
$, there exists a weak solution $u_{\varepsilon }\in W_{0}^{1,q+\varepsilon
}\left( \Omega \right) $ to the approximated Dirichlet problem (\ref%
{approximated Dirichlet problem}), i.e. 
\begin{align}
\int_{\Omega }\sum_{i=1}^{n}\Big\{ a^{i}\left( x, u_\varepsilon\left(
x\right) ,D u_\varepsilon\left( x\right) \right)& +\varepsilon \,\left(
1+\left\vert D u_\varepsilon\right\vert^{2}\right) ^{\frac{q+\varepsilon -2}{%
2}} (u_{\varepsilon })_{x_{i}}\Big\} v_{x_{i}}\,dx  \notag \\
& \qquad +\int_{\Omega }b\left( x\right) v\,dx=0\,  \label{e:wsol}
\end{align}%
for every $v\in W_{0}^{1,q+\varepsilon }(\Omega )$.
\end{remark}

\medbreak

\textit{Step 6:} In this step we prove Lemma \ref{global gradient estimate}.

By the assumption $p\geq 2$, the inequality (\ref{in the proof 4}) and
reasoning as in \cite[Lemma 4.4, eq. (4.12)]{Marcellini 1991}, there exists $%
\tilde{c}>0$ such that, for every $\xi,\eta\in \mathbb{R}^n$ and for every $%
u\in \mathbb{R}$, 
\begin{equation*}
|\xi |^{p}\leq \tilde{c}\left( |\eta |^{p}+\sum_{i=1}^{n}\left(
a^{i}(x,u,\xi )-a^{i}(x,u,\eta )\right) (\xi _{i}-\eta _{i})\right);
\end{equation*}
in particular, for $\eta=0$, 
\begin{equation}
|\xi |^{p}\leq \tilde{c}\left\{\left( a\left( x,u,\xi \right) ,\xi \right)-
\left( a\left( x,u,0 \right) ,\xi \right)\right\}.  \label{e:Marcellini1991}
\end{equation}%
We have 
\begin{align}
-\left( a\left( x,u,0 \right) ,\xi \right) =& -\left( a\left( x,0,0\right)
,\xi\right) -\int_{0}^{1}\sum_{i=1}^{n}\frac{d}{dt}a^{i}\left(
x,tu,0\right)\xi _{i}\,dt  \notag \\
=& -\left( a\left( x,0,0\right) ,\xi \right)
-\int_{0}^{1}\sum_{i=1}^{n}a_{u}^{i}\left( x,tu,0\right) \xi _{i} u \,dt\,.
\label{consequences 1}
\end{align}
Using (\ref{growth 1})$_2$ 
\begin{equation*}
\left\vert \int_{0}^{1}\sum_{i=1}^{n}a_{u}^{i}\left( x,tu,0\right) \xi _{i}
u\,dt\right\vert \leq \sqrt{n} M\left( |u|+\left\vert u\right\vert ^{\beta
}\right) |\xi |.
\end{equation*}
By Young's inequality, for every $\tau >0$ 
\begin{equation*}
\left( |u|+\left\vert u\right\vert ^{\beta }\right) |\xi | \newline
\leq \frac{p-1}{p\,\tau ^{\frac{p}{p-1}}}\left( |u|^{\frac{p}{p-1}%
}+\left\vert u\right\vert ^{\beta \frac{p}{p-1}}\right) +\frac{2\tau ^{p}}{p}%
|\xi |^{p}.
\end{equation*}%
Analogously, for every $\tau >0$ 
\begin{equation*}
-\left( a\left( x,0,0\right) ,\xi \right) \leq \frac{p-1}{p\,\tau ^{\frac{p}{%
p-1}}}|a(x,0,0)|^{\frac{p}{p-1}}+\frac{\tau^{p}}{p}|\xi |^{p}.
\end{equation*}
By \eqref{consequences 1} and the estimates above, 
%  taking into account that 
%\[ |\eta|^{p}
%\le 	(1+\left\vert \eta \right\vert ^{2})^{\frac{p(q-1)}{2(p-1)}},
% \] 
by choosing $\tau $ small enough we get 
\begin{equation*}
-\tilde{c}\left( a\left( x,u,0 \right) ,\xi \right) \leq \ \frac{1 }{2}|\xi
|^{p}+c\left\{|u|^{\frac{p}{p-1}} +\left\vert u\right\vert ^{\beta \frac{p}{%
p-1}}+\left\vert a\left( x,0,0\right) \right\vert ^{\frac{p}{p-1}} \right\} ,
\end{equation*}
where $\tilde{c}$ is the constant at the right hand-side of (\ref%
{e:Marcellini1991}) and $c_1$ Therefore, by (\ref{e:Marcellini1991}) we get 
\begin{equation*}
|\xi |^{p}\leq c\left\{\left( a\left( x,u,\xi \right) ,\xi \right) + |u|^{%
\frac{p}{p-1}}+\left\vert u\right\vert^{\beta \frac{p}{p-1}}+\left\vert
a\left( x,0,0\right) \right\vert ^{\frac{p}{p-1}} \right\}.
\end{equation*}
In particular, if we consider $\xi =Du_{\varepsilon }(x)$ and $%
u=u_{\varepsilon }(x)$ we get 
\begin{align}  \label{e:coercivityGC}
|Du_{\varepsilon }(x)|^{p}\leq & \,c\sum_{i=1}^{n}a^{i}(x,u_{\varepsilon
}(x),Du_{\varepsilon }(x))(u_{\varepsilon })_{x_{i}}(x) \\
& +c\left\{ |u_{\varepsilon }(x)|^{\frac{p}{p-1}}+|u_{\varepsilon
}(x)|^{\beta \frac{p}{p-1}}+\left\vert a\left( x,0,0\right) \right\vert ^{%
\frac{p}{p-1}}\right\}.  \notag
\end{align}%
By (\ref{e:wsol}), used with $v=u_{\varepsilon }$, we get 
%and $b\in L^{\frac{p}{p-1}%
%}(\Omega )$, see (\ref{growth 4}), that is analogous to (4.4) in \cite%
%{Marcellini 1991},
\begin{align*}
\int_{\Omega }\sum_{i=1}^{n}a^{i}\left( x,u,Du_{\varepsilon }\right)
(u_{\varepsilon })_{x_{i}}\,dx=& -\varepsilon \int_{\Omega }\,\left(
1+\left\vert Du_{\varepsilon }\right\vert ^{2}\right) ^{\frac{q+\varepsilon}{%
2}}\,dx-\int_{\Omega }b(x)u_{\varepsilon }\,dx\,,  \notag \\
\le & -\int_{\Omega }b(x)u_{\varepsilon }\,dx,
\end{align*}
therefore, using (\ref{e:coercivityGC}) 
\begin{align}  \label{e:Duepsilonp1}
\int_{\Omega }|Du_{\varepsilon}(x)|^p\,dx\le &-c\int_{\Omega
}b(x)u_{\varepsilon} \\
& +c\int_{\Omega }\left\{ |u_{\varepsilon}(x)|^{\frac{p}{p-1}}+
|u_{\varepsilon}(x)|^{\beta\frac{p}{p-1}}+\left| a\left(
x,0,0\right)\right|^{\frac{p}{p-1}} \right\} \,dx.  \notag
\end{align}
By H\"older inequality with exponent $\frac{p^*(p-1)}{p\beta}$, that is
greater than $1$ because by assumption $\beta <p-1$, and by Sobolev
embedding theorem, we get 
\begin{align*}
\int_{\Omega} |u_{\varepsilon}|^{\beta\frac{p}{p-1}}\,dx \leq & c \left(
\int_{\Omega}\left\vert u \right\vert ^{p^*} \, dx\right)^{\frac{\beta p}{%
p^*(p-1)}}|\Omega|^{1-\frac{\beta p}{ p^*(p-1)}} \\
\le & c_S \left(\int_{\Omega}\left\vert Du\right\vert^{p} \, dx\right)^{%
\frac{\beta }{p-1}}|\Omega|^{1-\frac{\beta p}{ p^*(p-1)}}.
\end{align*}
By Young's inequality with exponent $\frac{p-1}{\beta}$ we get that for
every $\delta>0$ 
\begin{equation*}
\int_{\Omega} |u_{\varepsilon}|^{\beta\frac{p}{p-1}}\,dx \le \frac{%
\beta\delta^{\frac{p-1}{\beta}}}{p-1} \int_{\Omega}\left\vert
Du_{\varepsilon}\right\vert^{p} \, dx + \frac{p-1-\beta}{(p-1)\delta^{\frac{%
p-1}{p-1-\beta}}}c_{S}^{\frac{p-1}{p-1-\beta}} |\Omega|^{\big(\frac{p-1}{%
\beta}\big)^{\prime}\big(1-\frac{\beta p}{p^*(p-1)}\big)}.
\end{equation*}
The same computations as above, with $\beta=1$, give 
\begin{equation*}
\int_{\Omega} |u_{\varepsilon}|^{\frac{p}{p-1}}\,dx \le \frac{\delta^{p-1}}{%
p-1} \int_{\Omega}\left\vert Du_{\varepsilon}\right\vert^{p} \, dx + \frac{%
p-2}{(p-1)\delta^{\frac{p-1}{p-2}}}c_{S}^{\frac{ p-1}{p-2}}
|\Omega|^{(p-1)^{\prime}\big(1-\frac{p}{p^*(p-1)}\big)}.
\end{equation*}
To estimate the first integral at the right hand side of (\ref{e:Duepsilonp1}%
) we remark that inequality (\ref{q/p bound}) implies, if $p<n$, that $q<p^*$%
; this last condition is trivially satisfied, if $p\ge n$, by choosing $p^*$
greater than $q$. Since by (\ref{growth 4})$_2$ it is $b\in
L^{q^{\prime}}(\Omega)$ we get $b\in L^{(p^*)^{\prime}}(\Omega)$. Hence the
H\"older inequality with exponent $p^*$ gives 
\begin{align*}
-\int_{\Omega }b(x)u_{\varepsilon} \,dx \le& \left(\int_{\Omega}
|u_{\varepsilon}|^{p^*}\,dx\right)^{\frac{1}{p^*}}\left(\int_{\Omega }|b|^{%
\frac{p^*}{p^*-1}} \,dx\right)^{1-\frac{1}{p^*}} \\
\le &\, c_S \left(\int_{\Omega}\left\vert Du_{\varepsilon}\right\vert^{p} \,
dx\right)^{\frac{1}{p}} \|b\|_{L^{\frac{p^*}{p^*-1}}(\Omega)};
\end{align*}
therefore, by Young's inequality with exponent $p$ we get that for every $%
\delta>0$ 
\begin{equation*}
-\int_{\Omega }b(x)u_{\varepsilon} \,dx \le c_S\frac{\delta^{p}}{p}
\int_{\Omega}\left\vert Du_{\varepsilon}\right\vert^{p} \, dx+c_S\frac{p-1}{%
p\delta^{\frac{p}{p-1}}} \|b\|_{L^{\frac{p^*}{p^*-1}}(\Omega)}^{\frac{p}{p-1}%
}.
\end{equation*}
Taking into account these estimates, starting from (\ref{e:Duepsilonp1}), by
choosing $\delta$ small enough, we get 
\begin{align*}
\|Du_{\varepsilon}\|_{L^p(\Omega)}^p \le\, & c \|a(\cdot,0,0)\|^{\frac{p}{p-1%
}}_{L^{\frac{p}{p-1}}(\Omega)} +c \|b\|_{L^{\frac{p^*}{p^*-1}}(\Omega)}^{%
\frac{p}{p-1}} \\
&+c|\Omega|^{\big(\frac{p-1}{\beta}\big)^{\prime}\big(1-\frac{\beta p}{%
p^*(p-1)}\big)}+c|\Omega|^{(p-1)^{\prime}\big(1-\frac{p}{p^*(p-1)}\big)} 
\notag
\end{align*}
and (\ref{e:Lpuniforme}) holds true.

\medbreak

\textit{Step 7:} $L^{\infty}_{\mathrm{loc}}$-boundedness uniformly w.r.t. $%
\varepsilon$. We distinguish two cases: $p>n$ and $p\le n$.

By the previous Step 6, the solution $u_{\varepsilon }$ of the approximating
problem is bounded in $W_{0}^{1,p}(\Omega )$ uniformly w.r.t. $\varepsilon
\in \left( 0,\varepsilon _{0}\right] $. If $p>n$ the Sobolev-Morrey
embedding Theorem also implies an $L^{\infty }(\Omega )$-bound of $%
u_{\varepsilon }$ uniform w.r.t. $\varepsilon $. If $p\leq n$ we can apply
the local boundedness result \cite[Theorem 4.2]{Cupini-Marcellini-Mascolo
regularity 2024} to get a $L_{\mathrm{loc}}^{\infty }(\Omega )$-bound of $%
u_{\varepsilon }$ uniform w.r.t. $\varepsilon $.

Therefore, in any case, we can apply the assumptions in (\ref{growth 2})
with $M(L)$ independent of $\varepsilon$.

\medbreak

\textit{Step 8:} Regularity. For the solution $u_{\varepsilon }$ to the
elliptic differential equation in (\ref{approximated Dirichlet problem}) we
use the interior $L_{\mathrm{loc}}^{\infty }\left( \Omega \right) -$gradient
estimates for $Du$ and the $L_{\mathrm{loc}}^{2}\left( \Omega \right) -$%
estimate for the $n\times n$ matrix $D^{2}u$ of the second derivatives, as
in \cite[Theorem 2.1]{Cupini-Marcellini-Mascolo regularity 2024}. Precisely,
under the ellipticity condition (\ref{ellipticity}) and the growth
conditions (\ref{growth 1}),(\ref{growth 2}),(\ref{growth 3}), with $\tfrac{%
q+\varepsilon _{0}}{p}<1+\tfrac{1}{n}$, then $u_{\varepsilon }\in L_{\mathrm{%
loc}}^{\infty }\left( \Omega \right) $ and, for every $\Omega ^{\prime
}\subset \subset \Omega $, the norm $\left\Vert u_{\varepsilon }\right\Vert
_{L^{\infty }\left( \Omega ^{\prime }\right) }$ is bounded uniformly with
respect to $\varepsilon \in \left( 0,\varepsilon _{0}\right] $. Moreover,
there exists a constant $c$ (which we continue to denote with the same
symbol), depending on the $L_{\mathrm{loc}}^{\infty }\left( \Omega ^{\prime
}\right) $ bound of $u_{\varepsilon }$ but not on $\varepsilon \in \left(
0,\varepsilon _{0}\right] $, and there exist a parameters $\alpha \geq 1$
such that 
\begin{equation}
\left\Vert Du_{\varepsilon }\left( x\right) \right\Vert _{L^{\infty }\left(
B_{\varrho }\right) }\leq \left( \tfrac{c}{\left( R-\varrho \right) ^{n}}%
\int_{B_{R}}(1+\left\vert Du_{\varepsilon }\left( x\right) \right\vert
^{2})^{\frac{p}{2}}\,dx\right) ^{\frac{\alpha }{p}}
\label{gradient bound in the proof}
\end{equation}%
\begin{equation*}
\underset{\text{for }n>2}{=}\;\left( \tfrac{c}{\left( R-\varrho \right) ^{n}}%
\left\Vert (1+\left\vert Du_{\varepsilon }\left( x\right) \right\vert ^{2})^{%
\frac{1}{2}}\right\Vert _{L^{p}\left( B_{R}\right) }^{p}\right) ^{\frac{2}{%
\left( n+2\right) p-n\left( q+\varepsilon \right) }}\,;
\end{equation*}%
and also, for the $n\times n$ matrix $D^{2}u$ of the second derivatives 
\begin{equation}
\int_{B_{\rho }}\left\vert D^{2}u_{\varepsilon }(x)\right\vert ^{2}\,dx\leq 
\tfrac{c}{\left( R-\rho \right) ^{2}}\int_{B_{R}}(1+\left\vert
Du_{\varepsilon }\left( x\right) \right\vert ^{2})^{\frac{q+\varepsilon }{2}%
}\,dx\,  \label{bound on second derivatives in the proof}
\end{equation}%
\begin{equation*}
\leq \frac{c^{\prime }}{\left( R-\rho \right) ^{2}}\left( \tfrac{1}{\left(
R-\varrho \right) ^{\gamma \vartheta \left( q+\varepsilon -p\right) }}%
\int_{B_{R}}(1+\left\vert Du_{\varepsilon }\left( x\right) \right\vert
^{2})^{\frac{p}{2}}\,dx\right) ^{\frac{\alpha \left( q+\varepsilon \right) }{%
\vartheta p}}\,.
\end{equation*}

\medbreak

\textit{Step 9:} Passage to the limit as $\varepsilon \rightarrow 0^{+}$. We
are in the conditions to use compactness. In fact, by combining (\ref%
{gradient bound in the proof}),(\ref{bound on second derivatives in the
proof}) with the global $L^{p}\left( \Omega \right) -$gradient bound of
Lemma \ref{global gradient estimate}, we obtain that, for every $\Omega
^{\prime }\subset \subset \Omega $ there exists a constant $c=c\left( \Omega
^{\prime }\right) $, depending on $\Omega ^{\prime }$ and on the $L_{\mathrm{%
loc}}^{\infty }\left( \Omega ^{\prime }\right) $ bound of $u_{\varepsilon }$%
, which is uniform with respect to $\varepsilon \in \left( 0,\varepsilon _{0}%
\right] $, such that 
\begin{equation}
\left\{ 
\begin{array}{l}
\left\Vert Du_{\varepsilon }\left( x\right) \right\Vert _{L^{p}\left( \Omega
\right) }\leq c\,, \\ 
\left\Vert u_{\varepsilon }\left( x\right) \right\Vert _{L^{\infty }\left(
\Omega ^{\prime }\right) }\leq c\left( \Omega ^{\prime }\right) \,, \\ 
\left\Vert Du_{\varepsilon }\left( x\right) \right\Vert _{L^{\infty }\left(
\Omega ^{\prime }\right) }\leq c\left( \Omega ^{\prime }\right) \,, \\ 
\left\Vert D^{2}u_{\varepsilon }(x)\right\Vert _{L^{2}\left( \Omega ^{\prime
}\right) }\leq c\left( \Omega ^{\prime }\right) \,.%
\end{array}%
\right.   \label{bounds for compacteness}
\end{equation}%
Therefore we can extract a sequence $u_{\varepsilon _{k}}$ with the
properties that, as $k\rightarrow +\infty $ the numerical sequence $%
\varepsilon _{k}\rightarrow 0^{+}$ and $u_{\varepsilon _{k}}$ converges to a
Sobolev function $u:\Omega \rightarrow \mathbb{R}$ in the following ways
(the symbol \textquotedblleft $\rightharpoonup $" as usual means \textit{%
weak convergence}; while $\rightharpoonup ^{\ast }$ means \textit{weak}$%
^{\ast }$\textit{\ convergence})%
\begin{equation}
\left\{ 
\begin{array}{l}
u_{\varepsilon _{k}}\overset{W_{0}^{1,\,p}\left( \Omega \right) }{%
\rightharpoonup }u\;, \\ 
u_{\varepsilon _{k}}\overset{W_{\mathrm{loc}}^{1,\infty }\left( \Omega
\right) }{\rightharpoonup ^{\ast }}u\;, \\ 
u_{\varepsilon _{k}}\overset{W_{\mathrm{loc}}^{2,2}\left( \Omega \right) }{%
\rightharpoonup }u\;, \\ 
u_{\varepsilon _{k}}\overset{W_{\mathrm{loc}}^{1,2}\left( \Omega \right) }{%
\rightarrow }u\;.%
\end{array}%
\right.   \label{convergence by compacteness}
\end{equation}%
In words, as $k\rightarrow +\infty $ the sequence $u_{\varepsilon _{k}}$
converges to the Sobolev function $u:\Omega \rightarrow \mathbb{R}$ in the
weak topology of $W_{0}^{1,\,p}\left( \Omega \right) $ and in the weak$%
^{\ast }$ topology of $W_{\mathrm{loc}}^{1,\infty }\left( \Omega \right) $;
moreover also in the weak topology of $W_{\mathrm{loc}}^{2,2}\left( \Omega
\right) $ and, by compactness, in the strong topology of $W_{\mathrm{loc}%
}^{1,2}\left( \Omega \right) $. All these conditions imply that 
\begin{equation}
u\in W_{0}^{1,\,p}\left( \Omega \right) \cap W_{\mathrm{loc}}^{1,\infty
}\left( \Omega \right) \cap W_{\mathrm{loc}}^{2,2}\left( \Omega \right) \,.
\label{functional spaces for u}
\end{equation}%
The bound in $W_{\mathrm{loc}}^{2,2}\left( \Omega \right) $, uniform with
respect to $\varepsilon \in \left( 0,\varepsilon _{0}\right] $, also implies
the convergence of the gradient $Du_{\varepsilon _{k}}\left( x\right) $ to
the gradient $Du\left( x\right) $ a.e. in $\Omega $. Therefore we can go to
the limit as $k\rightarrow +\infty $ in the integral form of the equation in
(\ref{approximated Dirichlet problem})$_{1}$ when $\varepsilon $ is replaced
by $\varepsilon _{k}$, i.e. 
\begin{equation*}
\int_{\Omega }\sum_{i=1}^{n}\left\{ a^{i}\left( x,u_{\varepsilon
_{k}},Du_{\varepsilon _{k}}\right) +\varepsilon _{k}\,\left( 1+\left\vert
Du_{\varepsilon _{k}}\right\vert ^{2}\right) ^{\frac{q+\varepsilon _{k}-2}{2}%
}\left( u_{\varepsilon _{k}}\right) _{x_{i}}\right\}
v_{x_{i}}\,dx\,+\int_{\Omega }bv\,dx=0\,,
\end{equation*}%
respectively, for every $k\in \mathbb{N}$, valid for all $v\in
W_{0}^{1,q+\varepsilon _{k}}(\Omega )$. To uniform the class of test
functions with respect to $k\in \mathbb{N}$, we first consider a generic
open set $\Omega ^{\prime }\subset \subset \Omega $ and a generic test
function $v\in C_{0}^{1}\left( \Omega ^{\prime }\right) $. Since $%
\varepsilon _{k}\rightarrow 0$, we obtain that $u$ satisfies the condition 
\begin{equation}
\int_{\Omega }\sum_{i=1}^{n}a^{i}\left( x,u,Du\right)
v_{x_{i}}\,dx\,+\int_{\Omega }bv\,dx=0\,,\;\;\forall \,v\in C_{0}^{1}\left(
\Omega ^{\prime }\right) \,.  \label{smooth test function}
\end{equation}%
By regularization, more generally (\ref{smooth test function}) holds for
every $v\in W_{0}^{1,q}\left( \Omega ^{\prime }\right) $ too. Therefore $%
u\in W_{0}^{1,\,p}\left( \Omega \right) \cap W_{\mathrm{loc}}^{1,\infty
}\left( \Omega \right) $ is a weak solution to the differential equation, as
defined in (\ref{definition of weak solution}) and it is also a solution to
the original Dirichlet problem (\ref{Dirichlet problem}).

Moreover we can go to the limit as $k\rightarrow +\infty $ in the right hand
side of (\ref{gradient bound in the proof}) and in the left hand side too by
lower semicontinuity. Similarly we can go to the limit as $k\rightarrow
+\infty $ also in the estimate (\ref{bound on second derivatives in the
proof}). We obtain for $u$ the validity of the estimates (\ref{gradient
estimate}),(\ref{second derivatives estimate}). The proof of Theorem \ref%
{Leray-Lions Theorem under p,q-growth} is complete.

\bigskip

\textbf{Acknowledgement} The authors are members of the \textit{Gruppo
Nazionale per l'Analisi Matematica, la Probabilit\`{a} e le loro
Applicazioni (GNAMPA)} of the \textit{Istituto Nazionale di Alta Matematica
(INdAM)}.


\bigskip

\begin{thebibliography}{99}
\bibitem{Abdellaoui-Dall'Aglio-Peral 2006} \textsc{B. Abdellaoui, A.
Dall'Aglio, I. Peral:} Some remarks on elliptic problems with critical
growth in the gradient, J. Differential Equations, \textbf{222} (2006),
21-62.

\bibitem{Alvino-Boccardo-Ferone-Orsina-Trombetti 2003} \textsc{A. Alvino, L.
Boccardo, V. Ferone, L. Orsina, G. Trombetti:} Existence results for
nonlinear elliptic equations with degenerate coercivity, Ann. Mat. Pura
Appl., \textbf{182} (2003), 53-79.

\bibitem{Alvino-Mercaldo-Volpicelli-Betta 2019} \textsc{A. Alvino, A.
Mercaldo, R. Volpicelli, M.F. Betta:}{\ }On a class of nonlinear elliptic
equations with lower order terms, Differential Integral Equations, \textbf{32%
} (2019), 223-232.

\bibitem{Arcoya-Carmona-Leonori-MartinezAparicio-Orsina-Petitta 2009} 
\textsc{D. Arcoya, J. Carmona, T. Leonori, P. Mart\'{\i}nez-Aparicio, L.
Orsina, F. Petitta:}{\ }Existence and nonexistence of solutions for singular
quadratic quasilinear equations, J. Differential Equations, \textbf{246}
(2009), 4006-4042.

\bibitem{Baroni-Colombo-Mingione 2018} \textsc{P. Baroni, M. Colombo, G.
Mingione:}{\ Regularity for general functionals with double phase,}\emph{\ }%
Calc. Var. Partial Differ. Equ., \textbf{57} (2018), 48 pp.

\bibitem{Boccardo-Croce 2010} \textsc{L. Boccardo, G. Croce:} Esistenza e
regolarit\`{a} di soluzioni per alcuni problemi ellittici, 2010, Pitagora
Ed., pp.112.

\bibitem{Boccardo-Murat-Puel 1988} \textsc{L. Boccardo, F. Murat, J.P. Puel:}
Existence of bounded solutions for nonlinear elliptic unilateral problems,
Ann. Mat. Pura Appl., \textbf{152} (1988), 183-196.

\bibitem{Boccardo-Murat-Puel 1992} \textsc{L. Boccardo, F. Murat, J.P. Puel:}
$L^{\infty }$ estimate for some nonlinear elliptic partial differential
equations and application to an existence result, SIAM J. Math. Anal., 
\textbf{23} (1992), 326-333.

\bibitem{Boegelein-Dacorogna-Duzaar-Marcellini-Scheven 2020} \textsc{V. B%
\"{o}gelein, B. Dacorogna, F. Duzaar, P. Marcellini, C. Scheven:} Integral
convexity and parabolic systems, \textit{SIAM Journal on Mathematical
Analysis (SIMA)}, \textbf{52} (2020), 1489-1525.

\bibitem{Boegelein-Duzaar-Marcellini-Scheven JMPA 2021} \textsc{V. B\"{o}%
gelein, F. Duzaar, P. Marcellini, C. Scheven:} Boundary regularity for
elliptic systems with $p,q-$growth, J. Math. Pures Appl., \textbf{159}
(2022), 250--293.

\bibitem{Byun-Oh 2020} \textsc{Sun-Sig Byun, Jehan Oh:} Regularity results
for generalized double phase functionals, Anal. PDE, \textbf{13} (2020),
1269-1300.

\bibitem{Chlebicka 2018} \textsc{I. Chlebicka:} A pocket guide to nonlinear
differential equations in Musielak--Orlicz spaces, \textit{Nonlinear Analysis%
}, \textbf{175 }(1918), 1-27.

\bibitem{Chlebicka-DeFilippis 2019} \textsc{I. Chlebicka, C. De Filippis:}
Removable sets in non-uniformly elliptic problems, \textit{Annali di
Matematica Pura ed Applicata}, \textbf{199} (2020), 619-649.

\bibitem{Cianchi-Mazya 2011} \textsc{A. Cianchi, V.G. Maz'ya:} Global
Lipschitz regularity for a class of quasilinear elliptic equations, Comm.
Partial Differential Equations, \textbf{36} (2011), 100-133.

\bibitem{Cianchi-Mazya 2014} \textsc{A. Cianchi, V.G. Maz'ya:} Global
boundedness of the gradient for a class of nonlinear elliptic systems, Arch.
Ration. Mech. Anal.\emph{,} \textbf{212} (2014), 129-177.

\bibitem{Colombo-Mingione 2015} \textsc{M. Colombo, G. Mingione:} Regularity
for double phase variational problems, Arch. Ration. Mech. Anal., \textbf{215%
} (2015), 443--496.

\bibitem{Cupini-Marcellini-Mascolo 2014} \textsc{G. Cupini, P. Marcellini,
E. Mascolo:} {Existence and regularity for elliptic equations under $p,q-$%
growth,} {\ Adv. Differential Equations},\text{\ }\textbf{19}{\ (2014),
693-724.}

\bibitem{Cupini-Marcellini-Mascolo 2017} \textsc{G. Cupini, P. Marcellini,
E. Mascolo:}{\ Regularity of minimizers under limit growth conditions,}
Nonlinear Anal., \textbf{153} (2017), 294-310.

\bibitem{Cupini-Marcellini-Mascolo 2018} \textsc{G. Cupini, P. Marcellini,
E. Mascolo:} \ Nonuniformly elliptic energy integrals with $p,q-$growth, {%
Nonlinear Anal.,} \textbf{177} (2018), 312-324.

\bibitem{Cupini-Marcellini-Mascolo 2023} \textsc{G. Cupini, P. Marcellini,
E. Mascolo:} Local boundedness of weak solutions to elliptic equations with $%
p,q-$growth, Math. Eng., \textbf{5} (2023), Paper No. 065, 28 pp. \
https://doi.org/10.3934/mine.2023065

\bibitem{Cupini-Marcellini-Mascolo regularity 2024} \textsc{G. Cupini, P.
Marcellini, E. Mascolo:} Regularity for nonuniformly elliptic equations with 
$p,q-$growth and explicit $x,u-$dependence, preprint 2023.

\bibitem{Cupini-Marcellini-Mascolo-Passarelli 2023} \textsc{G. Cupini, P.
Marcellini, E. Mascolo, A. Passarelli di Napoli:} Lipschitz regularity for
degenerate elliptic integrals with $p,q-$growth, Advances in Calculus of
Variations, \textbf{16} (2023), 443-465.

\bibitem{Dacorogna-Marcellini 1998} \textsc{B. Dacorogna, P. Marcellini:}
Cauchy--Dirichlet problem for first order nonlinear systems, J. Functional
Analysis, \textbf{152} (1998), 404-446.

\bibitem{De Filippis JMPA 2022} \textsc{C. De Filippis:} Quasiconvexity and
partial regularity via nonlinear potentials, J. Math. Pures Appl., \textbf{%
163} (2022), 11-82.

\bibitem{DeFilippis-Leonetti-Marcellini-Mascolo 2023} \textsc{F. De
Filippis, F. Leonetti, P. Marcellini, E. Mascolo:} The Sobolev class where a
weak solution is a local minimizer, Atti Accad. Naz. Lincei Rend. Lincei
Mat. Appl., 2023, to appear.

\bibitem{DeFilippis-Mingione 2020} \textsc{C. De Filippis, G. Mingione:} On
the regularity of minima of non-autonomous functionals, J. Geom. Anal., 
\textbf{30} (2020), 1584-1626.

\bibitem{DeFilippis-Mingione ARMA 2021} \textsc{C. De Filippis, G. Mingione:}
Lipschitz bounds and nonautonomous integrals, Arch. Ration. Mech. Anal., 
\textbf{242} (2021), 973-1057.

\bibitem{Defilippis-Piccinini 2022-2023} \textsc{C. De Filippis, M.
Piccinini:} Borderline global regularity for nonuniformly elliptic systems,
International Mathematics Research Notices, to appear.
https://doi.org/10.1093/imrn/rnac283

\bibitem{De Filippis-Stroffolini 2023} \textsc{C. De Filippis, B.
Stroffolini:} Singular multiple integrals and nonlinear potentials, Journal
of Functional Analysis, \textbf{285} (2023). \
https://www.sciencedirect.com/science/article/pii/S002212362300109X

\bibitem{Diening-Harjulehto-Hasto-Ruzicka 2011} \textsc{L. Diening, P.
Harjulehto, P. Hasto, M. Ruzicka:} Lebesgue and Sobolev Spaces with Variable
Exponents, Lecture Notes in Mathematics, \textbf{2017}, Springer,
Heidelberg, 2011.

\bibitem{DiMarco-Marcellini 2020} \textsc{T. Di Marco, P. Marcellini:}
A-priori gradient bound for elliptic systems under either slow or fast
growth conditions, Calc. Var. Partial Differential Equations, \textbf{59}
(2020), 26 pp.

\bibitem{Eleuteri-Marcellini-Mascolo 2019} \textsc{M. Eleuteri, P.
Marcellini, E. Mascolo:} Regularity for scalar integrals without structure
conditions, Adv. Calc. Var., \textbf{13} (2020) 279-300.

\bibitem{Eleuteri-Marcellini-Mascolo-Perrotta 2022} \textsc{M. Eleuteri, P.
Marcellini, E. Mascolo, S. Perrotta:} Local Lipschitz continuity for energy
integrals with slow growth, Ann. Mat. Pura Appl., \textbf{201} (2022),
1005-1032.

\bibitem{Eleuteri-Passarelli 2023} \textsc{M. Eleuteri, A. Passarelli di
Napoli:} Lipschitz regularity of minimizers of variational integrals with
variable exponents, Nonlinear Analysis: Real World Applications, \textbf{71}
(2023), p. 103815.

\bibitem{Fang-Radulescu-Zhang-Zhang 2022} \textsc{Yuzhou Fang, V.D. R\u{a}%
dulescu, Chao Zhang, Xia Zhang:} Gradient estimates for multi-phase problems
in Campanato spaces, Indiana University Mathematics Journal, \textbf{71}
(2022).

\bibitem{Ferone-Murat 2000} \textsc{V. Ferone, F. Murat:}{\ }Nonlinear
problems having natural growth in the gradient: an existence result when the
source terms are small, Nonlinear Anal., \textbf{42} (2000), 1309-1326.

\bibitem{Hartman-Stampacchia 1966} \textsc{P. Hartman, G. Stampacchia:}{\ }%
On some non-linear elliptic differential-functional equations, Acta Math., 
\textbf{115} (1966), 271-310.

\bibitem{Hasto-Ok 2022} \textsc{P. H\"{a}st\"{o}, J. Ok:}{\ }Maximal
regularity for local minimizers of non-autonomous functionals, J. Eur. Math.
Soc., \textbf{24} (2022), 1285-1334.

\bibitem{Le 2009} \textsc{V.K. Le:}{\ Some existence results and properties
of solutions in quasilinear variational inequalities with general growths,}
Differ. Equ. Dyn. Syst., \textbf{17} {(2009), 343-364}.

\bibitem{Leray-Lions 1965} \textsc{J. Leray, J.L. Lions:}{\ Quelques r\'{e}%
sultats de Vi\u{s}ik sur les probl\`{e}mes elliptiques nonlin\'{e}aires par
les m\'{e}thodes de Minty-Browder,} Bull. Soc. Math. France, \textbf{93} {%
(1965), 97-107}.

\bibitem{Lions 1969} \textsc{J.L. Lions:}{\ }Quelques m\'{e}thodes de r\'{e}%
solution des probl\`{e}mes aux limites non lin\'{e}aires, Dunod et Gauthier
Villars, Paris, 1969.

\bibitem{Liu-Pucci 2023} \textsc{J. Liu, P. Pucci:}{\ }Existence of
solutions for a double-phase variable exponent equation without the
Ambrosetti-Rabinowitz condition, Adv. Nonlinear Anal., \textbf{12} (2023),
Paper No. 20220292, 18 pp.

\bibitem{Marcellini 1984} \textsc{P. Marcellini:}{\ }Quasiconvex quadratic
forms in two dimensions, Applied Mathematics and Optimization, \textbf{11}
(1984), 183-189.

\bibitem{Marcellini ARMA 1989} \textsc{P. Marcellini:}{\ Regularity of
minimizers of integrals in the calculus of variations with non standard
growth conditions},\ Arch. Rational Mech. Anal.,\textbf{\ 105}{\ }(1989),
267-284.

\bibitem{Marcellini 1991} \textsc{P. Marcellini:}{\ Regularity and existence
of solutions of elliptic equations with }$p,q-$growth conditions,\ J.
Differential Equations,\textbf{\ 90}{\ (1991), 1-30}.

\bibitem{Marcellini 2020} \textsc{P. Marcellini:} Regularity under general
and $p,q-$growth conditions, Discrete Contin. Dyn. Syst. Ser. S, \textbf{13}
(2020), 2009-2031.

\bibitem{Marcellini 2023} \textsc{P. Marcellini:} Local Lipschitz continuity
for $p,q-$PDEs with explicit $u-$dependence, Nonlinear Anal., \textbf{226}
(2023), Paper No. 113066, 26 pp. https://doi.org/10.1016/j.na.2022.113066

\bibitem{Mihailescu-Pucci-Radulescu 2008} \textsc{M. Mihailescu, P. Pucci,
V. Radulescu:} Eigenvalue problems for anisotropic quasilinear elliptic
equations with variable exponent, J. Math. Anal. Appl., \textbf{340} (2008),
687--698.

\bibitem{Papageorgiou-Radulescu-Zhang 2022} \textsc{N.S. Papageorgiou, V.D. R%
\u{a}dulescu, Y. Zhang:} Resonant double phase equations, Nonlinear
Analysis: Real World Applications, \textbf{64} (2022), 103454, 20 pp.

\bibitem{Ragusa-Tachikawa 2020} \textsc{M.A. Ragusa, A. Tachikawa:}
Regularity for minimizers for functionals of double phase with variable
exponents, Adv. Nonlinear Anal., \textbf{9} (2020), 710--728.

\bibitem{Sanchon-Urbano 2009} \textsc{M. Sanch\'{o}n, J.M. Urbano:} Entropy
solutions for the $p(x)-$Laplace equation, Trans. Amer. Math. Soc., \textbf{%
361} (2009), 6387-6405
\end{thebibliography}
\end{document}